\documentclass[11pt]{article}
\usepackage{amsmath}
\usepackage{amsthm}
\usepackage{amsfonts}
\usepackage{amssymb}
\usepackage{enumerate}
\title{A theorem on shifted Hankel determinants}
\author{Mike Tyson\thanks{Email: mgtyson66@gmail.com}}
\date{}

\newtheorem{thm}{Theorem}

\begin{document}

\maketitle

Let $(\mu(n))_{n\ge 0}$ be a sequence with $\mu(0)=1$, and assume none of the Hankel determinants $H_n=\det(\mu(i+j))_{i,j=0}^{n-1}$ vanish. Then one can define the family of polynomials
\[
p_n(x)=\sum_{j=0}^n(-1)^{n-j}x^jp(n,j)=\frac{1}{H_n}\det\begin{pmatrix}
\mu(0) & \mu(1) & \cdots & \mu(n) \\
\vdots & \vdots & \ddots & \vdots \\
\mu(n-1) & \mu(n) & \cdots & \mu(2n-1) \\
1 & x & \cdots & x^n
\end{pmatrix},
\]
which are orthogonal with respect to the weight with $n$th moment $\mu(n)$. In \cite{2019arXiv190210468C}, Johann Cigler conjectured that
\[
\det(\mu(i+j))_{i,j=0}^{m-1}\cdot \det(p(i+m,j))_{i,j=0}^{n-1}=\det(\mu(i+j+n))_{i,j= 0}^{m-1}.
\]
The purpose of this note is to prove this conjecture.

For $0\le a\le j\le b$, let $D_j(\mu;a,b;\lambda)$ be the determinant of the $(b-a)$-by-$(b-a)$ matrix
\[
\begin{pmatrix}
	\mu(a) & \cdots & \mu(j-1) & \mu(j+1) & \cdots & \mu(b) \\
	\vdots & \ddots & \vdots & \vdots & \ddots & \vdots \\
	\mu(b-1) & \cdots & \mu(j+b-a-2) & \mu(j+b-a) & \cdots & \mu(2b-a-1)+\lambda
\end{pmatrix},
\]
where the $\lambda$ is added only to the bottom right entry. For $j<a$ or $b<j$ let $D_j(\mu;a,b;\lambda)=0$. We will also denote $D_j(\mu;a,b;0)$ by $D_j(\mu;a,b)$ and typically suppress the $\mu$. Note that when the $\mu(i)$ are left as variables and $a\le j\le b$, none of these determinants are identically zero. We will divide by finitely many of these determinants over the course of the proof, but the end result is a polynomial identity which therefore applies even when some $D_j(a,b;0)$ vanish.

Since $H_n=D_n(0,n)$ and $p(n,j)=D_j(0,n)/D_n(0,n)$, the conjecture is that $\det\left(D_j(0,i+m)\right)_{i,j=0}^{n-1}$ equals
\[
D_{m+n}(n,m+n)\cdot D_{m+1}(0,m+1)\cdots D_{m+n-1}(0,m+n-1).
\]
For instance, when $m=2$ and $n=3$ the conjecture is that
\[
\begin{pmatrix}
\begin{vmatrix}\mu(1) & \mu(2)\\\mu(2) & \mu(3)\end{vmatrix} & \begin{vmatrix}\mu(0) & \mu(2)\\\mu(1) & \mu(3)\end{vmatrix} & \begin{vmatrix}\mu(0) & \mu(1)\\\mu(1) & \mu(2)\end{vmatrix}\\
\begin{vmatrix}\mu(1) & \mu(2) & \mu(3)\\\mu(2) & \mu(3) & \mu(4)\\\mu(3) & \mu(4) & \mu(5)\end{vmatrix} & \begin{vmatrix}\mu(0) & \mu(2) & \mu(3)\\\mu(1) & \mu(3) & \mu(4)\\\mu(2) & \mu(4) & \mu(5)\end{vmatrix} & \begin{vmatrix}\mu(0) & \mu(1) & \mu(3)\\\mu(1) & \mu(2) & \mu(4)\\\mu(2) & \mu(3) & \mu(5)\end{vmatrix}\\
\begin{vmatrix}\mu(1) & \mu(2) & \mu(3) & \mu(4)\\\mu(2) & \mu(3) & \mu(4) & \mu(5)\\\mu(3) & \mu(4) & \mu(5) & \mu(6)\\\mu(4) & \mu(5) & \mu(6) & \mu(7)\end{vmatrix} & \begin{vmatrix}\mu(0) & \mu(2) & \mu(3) & \mu(4)\\\mu(1) & \mu(3) & \mu(4) & \mu(5)\\\mu(2) & \mu(4) & \mu(5) & \mu(6)\\\mu(3) & \mu(5) & \mu(6) & \mu(7)\end{vmatrix} & \begin{vmatrix}\mu(0) & \mu(1) & \mu(3) & \mu(4)\\\mu(1) & \mu(2) & \mu(4) & \mu(5)\\\mu(2) & \mu(3) & \mu(5) & \mu(6)\\\mu(3) & \mu(4) & \mu(6) & \mu(7)\end{vmatrix}
\end{pmatrix}
\]
has determinant
\[
\begin{vmatrix}\mu(3)& \mu(4)\\\mu(4) & \mu(5)\end{vmatrix}\begin{vmatrix}\mu(0)& \mu(1) & \mu(2)\\\mu(1) & \mu(2) & \mu(3)\\\mu(2) & \mu(3) & \mu(4)\end{vmatrix}\begin{vmatrix}\mu(0)& \mu(1) & \mu(2) & \mu(3)\\\mu(1) & \mu(2) & \mu(3) &\mu(4)\\\mu(2) & \mu(3) & \mu(4) & \mu(5)\\\mu(3) & \mu(4) & \mu(5) & \mu(6)\end{vmatrix}.
\]

The conjecture can be proven by induction on $n$ via the following theorem.

\begin{thm}
	Let $\nu(i)=\mu(i+1)$ for all $i\ge 0$. Then $\det\left(D_j(\mu;0,i+m)\right)_{i,j=0}^{n-1}$ equals
	\[
	D_0(\mu;0,m)\cdot \det\left(D_j(\nu;0,i+m)\right)_{i,j=0}^{n-2}\cdot\prod_{i=1}^{n-1}\frac{D_{i+m}(\mu;0,i+m)}{D_{i+m}(\mu;1,i+m)}.
	\]
\end{thm}
\begin{proof}
	The $n=1$ case is trivial. Assume $n>1$ and let $Q$ be the matrix $\left(D_j(0,i+m)\right)_{i,j=0}^{n-1}$. Note that for $i\ge 1$ and $j<i+m$, the entry $Q_{i-1,j}$ appears as a cofactor in the expansion of $Q_{i,j}$. For $j\ge i+m$, $Q_{i-1,j}=0$. This means that adding $\lambda_i$ times the $(i-1)$th row to the $i$th row has the effect of replacing $Q_{i,j}=D_j(0,i+m)$ with $Q'_{i,j}=D_j(0,i+m;\lambda_i)$ for $j<i+m$. Since the entry $Q_{i-1,j}$ is generically nonzero, $\lambda_i$ can be chosen to make $D_0(0,i+m;\lambda_i)=0$. Perform this row operation on row $n-1$, then row $n-2$, and so on up to row $1$. Call the new matrix $Q'$.
	
	For $i\ge 1$, $Q'_{i,0}=0$, which (together with the nonvanishing of $Q_{i-1,0}$) implies that the last column of the underlying matrix of $D_0(0,i+m;\lambda_i)$ can be written as a linear combination of the first $i+m-1$ columns. That is, there are $a_1,\dots,a_{i+m-1}$ such that
	\[
	a_1\begin{pmatrix}\mu(1)\\ \vdots \\ \mu(i+m)\end{pmatrix}+\cdots+a_{i+m-1}\begin{pmatrix}\mu(i+m-1)\\ \vdots \\ \mu(2i+2m-2)\end{pmatrix}=\begin{pmatrix}\mu(i+m)\\ \vdots \\ \mu(2i+2m-1)+\lambda_i\end{pmatrix}.
	\]
	By restricting to the first $i+m-1$ coordinates of this equation and applying Cramer's rule, one finds that
	\[
	a_k=(-1)^{k+i+m-1}D_k(1,i+m)/D_{i+m}(1,i+m).
	\]
	
	For $j<i+m$, the solved-for column $(\mu(i+m),\dots,\mu(2i+2m-1)+\lambda_i)^\top$ appears in the underlying matrix of $D_j(0,i+m;\lambda_i)=Q'_{i,j}$. Substitute in the linear combination and perform column operations to remove summands of repeated columns. One is left with
	\[
	Q'_{i,j}=(-1)^{j+i+m-1}a_jD_{i+m}(0,i+m)=D_j(1,i+m)\frac{D_{i+m}(0,i+m)}{D_{i+m}(1,i+m)}.
	\]
	For $j=i+m$, $Q'_{i,j}$ was originally $D_{i+m}(0,i+m)$, so the above equation still holds. For $j>i+m$, both $Q'_{i,j}=D_j(0,i+m)$ and $D_j(1,i+m)$ are $0$, so the above equation again holds.
	
	Factor out $D_{i+m}(0,i+m)/D_{i+m}(1,i+m)$ from row $i$ of $Q'$ for each $i\ge 1$ to get a new matrix $Q''$. Like $Q'$, all entries in the leftmost column of $Q''$ are $0$ except for the top entry, which is $D_0(0,m)$. The bottom right $(n-1)$-by-$(n-1)$ submatrix of $Q''$ is $\left(D_{j}(1,i+m)\right)_{i,j=1}^{n-1}$, or $\left(D_j(\nu;0,i+m)\right)_{i,j=0}^{n-2}$. Therefore the determinant is as claimed.
\end{proof}

As an example, consider when $\mu(n)=C_n=\frac{1}{n+1}\binom{2n}{n}$ is a Catalan number. It can be verified via orthogonality that
\[
p_n(x)=\sum_{j=0}^n(-1)^{n-j}\binom{n+j}{n-j}x^j.
\]
As a result,
\[
\det\left(\binom{i+j+m}{i-j+m}\right)_{i,j=0}^{n-1}=\frac{\det(C_{n+i+j})_{i,j=0}^{m-1}}{\det(C_{i+j})_{i,j=0}^{m-1}}.
\]
The quotient on the right can be calculated with Dodgson condensation \cite{1999math......2004K} to be
\[
\prod_{1\le i\le j\le n-1}\frac{2m+i+j}{i+j}.
\]

\bibliographystyle{plain}
\bibliography{shiftedmoments}

\end{document}